\documentclass[reqno,b5paper]{amsart}

\usepackage{amsmath}
\usepackage{amssymb}
\usepackage{amsthm}
\usepackage{enumerate}
\usepackage[mathscr]{eucal}
\usepackage{eqlist}

\theoremstyle{plain}
\newtheorem{theorem}{Theorem}[section]
\newtheorem{prop}[theorem]{Proposition}
\newtheorem{lemma}{Lemma}[section]

\theoremstyle{definition}
\newtheorem{definition}{Definition}[section]

\theoremstyle{remark}
\newtheorem{example}{Example}

\numberwithin{equation}{section}

\begin{document}
\title[A generalization of cyclic amenability of Banach algebras]%
{A generalization of cyclic amenability of Banach algebras}
\author[Behrouz Shojaee \and Abasalt Bodaghi]%
{Behrouz Shojaee* \and Abasalt Bodaghi**}

\newcommand{\acr}{\newline\indent}

\address{\llap{*\,}Department of Mathematics\acr
                   Karaj Branch\acr
                   Islamic Azad University\acr
                   Karaj, IRAN}
\email{shoujaei@kiau.ac.ir}

\address{\llap{**\,}Department of Mathematics\acr
                   Garmsar Branch\acr
                   Islamic Azad University\acr
                   Garmsar, IRAN}
\email{abasalt.bodaghi@gmail.com}
\thanks{}


\begin{abstract}
This paper continues the investigation of Esslamzadeh and the
first author which was begun in \cite{ess}. It is shown that 
homomorphic image of an approximately cyclic amenable Banach
algebra is again approximately cyclic amenable. Equivalence
of approximate cyclic amenability of a Banach algebra
$\mathcal{A}$ and approximate cyclic amenability of
$M_{n}(\mathcal{A})$ is proved. It is shown that under certain
conditions the approximate cyclic amenability of second dual
$\mathcal{A}^{**}$ implies the approximate cyclic amenability of
$\mathcal{A}$.
\end{abstract}

\maketitle

\section{Introduction}
The concept of an amenable Banach algebra was defined and studied for the first time by Johnson in \cite{joh}.
Since then, several modifications of this notion have been
introduced by different authors (for instance,
\cite{bem} and \cite{gha2}) . The concept of cyclic amenability
was presented by Gronbeak in \cite{gro}. He investigated the
hereditary properties of this concept, found some relations
between cyclic amenability of a Banach algebra and the
trace extension property of its ideals. The notion of
approximate amenability was introduced by Ghahramani and Loy
\cite{gha2} for Banach algebras were they characterized the structure
of approximately amenable Banach algebras in several ways. They also
gave some examples of approximately amenable, non-amenable Banach
algebras to show that two notions of approximate amenability and
Johnson's amenability do not coincide (for more information and
examples refer also to \cite{gha3}).

 In this paper we define
the concept of approximate cyclic amenability for Banach algebras
and investigate the hereditary properties for this new notion.
Furthermore, we show that for Banach algebras $\mathcal{A}$ and
$\mathcal{B}$, if direct sum $\mathcal{A}\oplus \mathcal{B}$ with
$\ell^{1}$-norm is approximately cyclic amenable, then so are
$\mathcal{A}$ and $\mathcal{B}$. By the means of an example
we show that the converse is not true. However, the converse can
be held if $\mathcal{A}^2$ is dense in $\mathcal{A}$. We also
portray that approximate cyclic amenability of a Banach algebra
$\mathcal{A}$ is equivalent to the approximate cyclic amenability
of $M_{n}(\mathcal{A})$. In the third section, we show that
homomorphic image of an approximately cyclic amenable Banach
algebra under a continuous homomorphism is also approximately
cyclic amenable. As a consequence, we prove that if the tensor
product $\mathcal{A}\widehat{\otimes}\mathcal{B}$ is
approximately cyclic amenable, then so are $\mathcal{A}$ and
$\mathcal{B}$ provided that $\mathcal{A}$ and
$\mathcal{B}$ admit nonzero character. Finally, some mild conditions can be imposed on
$\mathcal{A}$ such that the approximate cyclic amenability of
$\mathcal{A}^{**}$ with the first or the second Arens products,
implies the approximate cyclic amenability of $\mathcal{A}$.

 Let $\mathcal{A}$ be a Banach algebra and $\mathcal{X}$ be a Banach
$\mathcal{A}$-bimodule. Then $\mathcal{X}^{*}$ is a Banach
$\mathcal{A}$-bimodule with module actions
$$\langle a\cdot x^{*},x\rangle=\langle x^{*},x\cdot a\rangle\quad,\quad \langle x^{*}\cdot a,x\rangle=\langle x^{*},a\cdot x\rangle\quad (a\in A, x\in X, x^{*}\in X^{*}).$$

A {\it derivation} from a Banach algebra $\mathcal{A}$ into a
Banach $\mathcal{A}$-bimodule $\mathcal{X}$ is a bounded linear
mapping $D:\mathcal{A}\longrightarrow \mathcal{X}$ such that
$D(ab)=D(a)\cdot b+a\cdot D(b)$ for every $a,b\in A$. A derivation
$D:\mathcal{A}\longrightarrow \mathcal{X}$ is called {\it inner}
if there exists $x\in X$ such that $D(a)=a\cdot x-x\cdot
a=\delta_{x}(a)\quad (a\in A)$. A Banach algebra $\mathcal{A}$
is called {\it amenable} if every bounded derivation
$D:\mathcal{A}\longrightarrow \mathcal{X}^{*}$ is inner for every Banach $\mathcal A$-bimodule $\mathcal X$. A
bounded derivation $D:\mathcal{A}\longrightarrow \mathcal{A}^{*}$
is called {\it cyclic} if $\langle Da,b\rangle+\langle
Db,a\rangle=0$ for all $a,b\in \mathcal{A}$. A Banach algebra
$\mathcal{A}$ is said to be {\it cyclic amenable} if every cyclic
bounded derivation $D:\mathcal{A}\longrightarrow \mathcal{A}^{*}$
is inner. A derivation $D:\mathcal{A}\longrightarrow \mathcal{X}$
is called {\it approximately inner} if there exists a net
$(x_{\alpha
 })\subseteq\mathcal{X}$ such that $$D(a)=\lim_{\alpha}(a\cdot x_{\alpha}-x_{\alpha}\cdot a)\quad (a\in \mathcal{A}).$$

Let $\mathcal{A}$ be an arbitrary Banach algebra. The first and
second Arens multiplications on $\mathcal{A}^{**}$  which are
denoted by $``\, \square\, "$ and $``\,\lozenge\,"$ respectively,
are defined in three steps. For every $a,b \in
\mathcal{A},a^{*}\in \mathcal{A}^{*}$ and $a^{**}, b^{**}\in
A^{**}$, the elements $a^{*}\cdot a, a\cdot a^{*}, a^{**}\cdot
a^{*}, a^{*}\cdot a^{**}$ of $\mathcal{A}^{*}$ and $a^{**}\square
b^{**}, a^{**}\lozenge b^{**}$ of $\mathcal{A}^{**}$ are defined
in the following way:
$$\langle a^{*}\cdot a,b\rangle=\langle a^{*},ab\rangle , \quad \langle a\cdot a^{*},b\rangle=\langle a^{*},ba\rangle$$
$$\langle a^{**}\cdot a^{*},a\rangle=\langle a^{**},a^{*}\cdot a\rangle,\quad \langle a^{*}\cdot a^{**},b\rangle=\langle a^{**},b\cdot a^{*}\rangle$$
$$\langle a^{**}\square b^{**},a^{*}\rangle =\langle a^{**},b^{**}\cdot a^{*}\rangle, \quad \langle a^{**}\lozenge b^{**},a^{*}\rangle=\langle b^{**},a^{*}\cdot a^{**}\rangle.$$
When these two products coincide on $\mathcal A^{**}$, we say
that $\mathcal A$ is {\it Arens regular} (for more details refer
to \cite{dal1}). 
 
 
\section{Approximate cyclic amenability }

We first recall the relevant material from \cite{ess}, thus making
our exposition self-contained.

\begin{definition}  A Banach
algebra $\mathcal{A}$ is called approximately cyclic amenable if
every cyclic derivation $D:\mathcal{A}\longrightarrow
\mathcal{A}^{*}$ is approximately inner.
\end{definition}

It is shown in  \cite[Example 4.3]{ess} that there is an
approximately cyclic amenable Banach algebra which is not cyclic
amenable. So the distinction between the cyclic amenability and the
approximate cyclic amenability of Banach algebras are followed
(see also Example 2.5). \vspace {0.2 cm}

Let $\mathcal{A}$ be a non-unital algebra. We denote by
$\mathcal{A}^{\#}$, the unitization algebra of $\mathcal{A}$,
formed by adjoining an identity to $\mathcal{A}$ so that
$\mathcal{A}^{\#}=\mathcal{A}\oplus\Bbb C$, with the product
$$(a,\alpha)(b,\beta)=(ab+\beta a+\alpha b,\alpha\beta)\quad\quad (a,b\in \mathcal{A},\quad \alpha,\beta\in \Bbb C).$$

In the case where  $\mathcal{A}$ is a Banach algebra,
$\mathcal{A}^{\#}$ is also a Banach algebra which contains
$\mathcal{A}$ as a closed ideal. The following result shows the
relationship of their approximate cyclic amenability which is
proved in \cite[Proposition 4.1]{ess}. 

\begin{prop} \label{ropo}
Let  $\mathcal{A}$ be a non-unital
Banach algebra. The unitization  algebra $\mathcal{A}^{\#}$ is
approximately cyclic amenable if and only if $\mathcal{A}$ is
approximately cyclic amenable.
\end{prop}

Let $\mathcal{I}$ be a closed ideal in $\mathcal{A}$. We say that
$\mathcal{I}$ has {\it the approximate trace extension property}
if for each $a^*\in \mathcal{I}^*$ with $a\cdot a^*=a^*\cdot a\,
(a\in\mathcal{A})$ there is a net $(a_{\alpha}^*)\subseteq
\mathcal{A}^*$ such that $$a_{\alpha}^*\mid_{\mathcal{I}}=a^*
(\text{for any $\alpha$})\qquad \text{and}\qquad a\cdot
a_{\alpha}^*-a_{\alpha}^*\cdot a\longrightarrow 0\,
\quad(a\in\mathcal{A}).$$

We also say that a bounded approximate identity $\{e_{\alpha}\}$
of $\mathcal{I}$ is {\it quasi central} for $\mathcal{A}$ if
$\lim_{\alpha}\|ae_{\alpha}-e_{\alpha}a\|=0$ for all
$a\in\mathcal{A}$.

\begin{prop} \label{ppp}
Let $\mathcal{A}$ be a Banach
algebra with a closed ideal $\mathcal{I}$.
\begin{enumerate}
\item[\emph{(i)}] {Suppose that $\mathcal{A}/\mathcal{I}$ is approximately cyclic amenable. Then $\mathcal{I}$ has the approximate trace extension
property;}
\item[\emph{(ii)}] {Suppose that $\mathcal{A}$ is cyclic amenable and $\mathcal{I}$ has the approximate trace extension property.
Then $\mathcal{A}/\mathcal{I}$ is approximately cyclic amenable;}
\item[\emph{(iii)}] {If $\mathcal{A}/\mathcal{I}$ is approximately cyclic amenable, $\overline{\mathcal{I}^2}=\mathcal{I}$, and $\mathcal{I}$ is cyclic amenable,
then $\mathcal{A}$ is approximately cyclic amenable;}
\item[\emph{(iv)}] {Suppose that $\mathcal{A}$ is approximately cyclic amenable
and $\mathcal{I}$ has a quasi-central bounded approximate
identity for $\mathcal{A}$. Then $\mathcal{I}$ is approximately
cyclic amenable.}
\end{enumerate}
\end{prop}
\begin{proof}
The proof of all parts is similar to
\cite[Proposition 2.2]{ess}.\end{proof}

Example 4.4 of \cite{ess} shows that the condition
$\overline{\mathcal{I}^2}=\mathcal{I}$ is necessary in Proposition
\ref{ppp} (iii). So this condition can not be removed.

\begin{theorem} \label{te}
Let $\mathcal{A}$ and $\mathcal{B}$ be
Banach algebras. Then the following statements hold:
\begin{enumerate}
\item[\emph{(i)}] {Suppose  $\mathcal{A}$ is commutative and
$\overline{\mathcal{A}^{2}}=\mathcal{A}$. If $\mathcal{A}$ and
$\mathcal{B}$ are approximately cyclic amenable, then
$\mathcal{A}\oplus\mathcal{B}$ is approximately cyclic amenable;}
\item[\emph{(ii)}] {If $\mathcal{A}\oplus\mathcal{B}$ is approximately cyclic amenable, then both $\mathcal{A}$ and $\mathcal{B}$ are approximately cyclic
amenable.}
\end{enumerate}
\end{theorem}
\begin{proof} (i): Since in the commutative case cyclic
amenability and approximate cyclic amenability of Banach algebras
are the same, the result follows from Proposition \ref{ppp} (iii).

(ii): Without loss of generality and in view of Proposition
\ref{ropo}, we see that $\mathcal{A}^{\#}\oplus\mathcal{B}^{\#}$ is
approximately cyclic amenable. Moreover, $\mathcal{A}^{\#}$ and
$\mathcal{B}^{\#}$ are ideals in
$\mathcal{A}^{\#}\oplus\mathcal{B}^{\#}$ that have quasi central
bounded approximate identities. Now by Proposition \ref{ppp} (iv),
$\mathcal{A}^{\#}$ and $\mathcal{B}^{\#}$ and as a result of
Proposition \ref{ropo}, $\mathcal{A}$ and $\mathcal{B}$ are
approximately cyclic amenable.
\end{proof}

In the general case, the converse of Theorem \ref{te} (ii) is not
true. Indeed, the condition
$\overline{\mathcal{A}^{2}}=\mathcal{A}$ in part (i) is
necessary as we will see in the following example.

Recall that a character on $\mathcal{A}$ is a non-zero
homomorphism from $\mathcal{A}$ into $\Bbb C$. The set of
characters on $\mathcal{A}$ is called the character space of
$\mathcal{A}$ and is denoted by $\Phi_{\mathcal{A}}$. Let
$\varphi\in \Phi_{\mathcal{A}}\cup\{0\}$. Then a linear functional
$d\in\mathcal{A}^*$ is a point derivation at $\varphi$ if
$$d(ab)=\varphi(a)d(b)+\varphi(b)d(a)\quad (a,b\in \mathcal{A}).$$

\begin{example}(i) Let $\mathcal{A}$ be a non-zero Banach
algebra with zero product, that is $ab=0$ for all $a,b\in
\mathcal{A}$. It is shown in \cite[Example 2.5]{gro} that such Banach algebra is cyclicly amenable if and only if its dimension is one. Consider $\mathcal{A}=\Bbb C$  with the zero product. Then $\mathcal{A}$ is approximately cyclic amenable. Since the
properties of cyclic amenability and  approximate cyclic amenability are the same for commutative Banach algebras, $\mathcal{A}\oplus\mathcal{A}$ is not approximately cyclic amenable.

(ii) Let $\mathcal{A}$ be a [approximately] weakly amenable Banach algebra. Then
$\overline{\mathcal{A}^{2}}=\mathcal{A}$. But from
part (i) we conclude  that in the general case this condition is
not necessary when $\mathcal{A}$ is [approximately] cyclic
amenable. Meanwhile, we have seen that $\mathcal{A}$ is
[approximately] cyclicly amenable, while it is not [approximately]
weakly amenable.

(iii) It is known  that if $\mathcal{A}$ [approximately] weakly
amenable then there is no non-zero continuous point derivation
on $\mathcal{A}$. This is not true for
[approximate] cyclic amenability. In other words, if
$\mathcal{A}=\Bbb C$ with the zero product, then for every
non-zero map $d\in \mathcal{A}^{*}$ and zero map $\varphi$, $d$
is a non-zero continuous point derivation at $\varphi$. However $\mathcal{A}$ is [approximately] cyclic amenable.
\end{example}

Let $\mathcal{A}$ be a Banach algebra, $\mathcal{ X}$ be a Banach
$\mathcal{A}$-bimodule and $n\in\Bbb N$. We shall regard
$M_{n}(\mathcal{\mathcal{X}})$ as a Banach
$M_{n}(\mathcal{A})$-bimodule in the obvious way so that
$$(a\cdot x)_{ij}=\sum_{k=1}^{n}a_{ik}\cdot x_{kj}\quad\quad (a=(a_{ij})\in M_{n}(\mathcal{A}), x=(x_{ij})\in M_{n}(\mathcal{X})).$$

We have the identity
$$M_{n}(\mathcal{X})^{*}=M_{n}(\mathcal{X}^{*})=\mathcal{X}^{*}\widehat{\otimes}M_{n}$$
with duality
$$\langle\Lambda,x\rangle=\sum_{i,j=1}^{n}\langle\lambda_{ij},x_{ij}\rangle, \quad\quad
(x=(x_{ij})\in M_{n}(\mathcal{X}),\Lambda=(\lambda_{ij})\in M_{n}(\mathcal{X}^*)).$$

Note that

\begin{equation}\label{ee1}(a\cdot\Lambda)_{ij}=\sum_{k=1}^{n}a_{jk}\cdot\lambda_{ik}\quad
\text {and}\quad
(\Lambda\cdot a)_{ij}=\sum_{k=1}^{n}\lambda_{kj}\cdot a_{ki}.
\end{equation}
for $a=(a_{ij})\in M_{n}(\mathcal{A})$ and
$\Lambda=(\lambda_{ij})\in M_{n}(\mathcal{X})^{*}$. One should
remember that $M_{n}(\mathcal{A})$ is isometrically algebra
isomorphic to $M_{n}\widehat{\otimes}\mathcal{A}$.

\begin{theorem}\label{thmm}
Let $\mathcal{A}$ be a Banach algebra
and $n\in\Bbb N$. Then $\mathcal{A}$ is approximately cyclic
amenable if and only if $M_{n}(\mathcal{A})$ is approximately
cyclic amenable.
\end{theorem}
\begin{proof}If $\mathcal{A}$ does not have an identity, according
to equality $M_{n}(\mathcal{A})^{\#}=M_{n}(\mathcal{A}^{\#})$ and
Proposition \ref{ropo}, we can assume that $\mathcal{A}$ has an identity.
Let $\mathcal{A}$ be an approximately cyclic amenable Banach
algebra and
 $D:M_{n}(\mathcal{A})\longrightarrow M_{n}(\mathcal{A})^{*}$ be a cyclic continuous
derivation. We regard $M_{n}$ as a subalgebra of $M_{n}(\mathcal{A})$.
Since $M_{n}$ is amenable, there exists an element
$\Lambda=(\lambda_{ij})\in M_{n}(\mathcal{A}^{*})$ with
$D|_{M_{n}}=ad_{\Lambda}$. Replacing D by $D-ad_{\Lambda}$, we
may suppose $D|_{M_{n}}=0$. For $r,s\in\Bbb N_{n}$ and $a\in \mathcal{A}$, let
$$D((a)_{rs})=(d_{ij}^{(r,s)}:i,j\in\Bbb N_{n})\in M_{n}(\mathcal{A}^{*}).$$

We have

$$D((a)_{rs})=D(\varepsilon_{r1}(a)_{11}\varepsilon_{1s})=\varepsilon_{r1}\cdot D((a)_{11})\cdot\varepsilon_{1s}.$$

Since $D(\varepsilon_{r1})=D(\varepsilon_{1s})=0$, by (\ref{ee1}) we
have $d_{ij}^{(r,s)}=0$ $(i,j \in \Bbb N_{n})$ except when
$(i,j)=(r,s)$ and in this case $d_{sr}^{(r,s)}=d_{11}^{(1,1)}$.
Putting $d_{11}^{(1,1)}=d(a)$, the map
$$d:\mathcal{A}\longrightarrow \mathcal{A}^{*}; a\mapsto d(a)$$

is a cyclic continuous derivation because for $a,b\in
\mathcal{A}$, by (\ref{ee1}) we get
$$D((ab)_{rs})=D((a)_{r1}\cdot (b)_{1s})=D((a)_{r1})\cdot (b)_{1s}+(a)_{r1}\cdot D((b)_{1s}).$$

This implies that $d(ab)=d(a)\cdot b+a\cdot d(b)$. On the other
hand,
$$\langle d(a),b\rangle+\langle d(b),a\rangle=\langle D((a)_{sr}),(b)_{rs}\rangle+\langle D((b)_{rs}),(a)_{sr}\rangle=0,$$

for all $a,b\in \mathcal{A}$. Due to approximate cyclic
amenability of $\mathcal{A}$, there exists a net
$(\lambda_{\alpha})\subseteq \mathcal{A}^{*}$ such that
$$d(a)=\lim_{\alpha}a\cdot\lambda_{\alpha}-\lambda_{\alpha}\cdot a\quad(a\in \mathcal{A}).$$

Take $\Lambda_{\alpha}\in M_{n}(\mathcal{A}^{*})$ to be the
matrix that has $\lambda_{\alpha}$ in each diagonal position and
zero elsewhere. Then by (\ref{ee1}) we see
$$D((a_{ij}))=\lim_{\alpha}(a_{ij})\cdot\Lambda_{\alpha}-\Lambda_{\alpha}\cdot (a_{ij})\quad\quad ((a_{ij})\in M_{n}(\mathcal{A}^{*})).$$

This shows that $M_{n}(\mathcal{A})$ is approximately
cyclic amenable.

Conversely,  suppose that $M_{n}(\mathcal{A})$ is approximately
cyclic amenable and $D:\mathcal{A}\longrightarrow
\mathcal{A}^{*}$ is a continuous cyclic derivation. It is easy to check
that $D\otimes1:\mathcal{A}\widehat{\otimes} M_{n}\longrightarrow
\mathcal{A}^{*}\widehat{\otimes} M_{n}=\mathcal{A}^{*}\otimes
M_{n}$ is a continuous cyclic derivation. Therefore there exists
a net $(a_{\alpha})\subseteq \mathcal{A}^{*}\otimes M_{n}$  such
that
$a_{\alpha}=\sum_{i,j=1}^{n}a_{ij}^{\alpha}\otimes\varepsilon_{ij}$
and
$$D\otimes1(\mathcal{B})=\lim_{\alpha}(\mathcal{B}\cdot a_{\alpha}-a_{\alpha}\cdot\mathcal{B}) \quad\quad \mathcal{B}\in \mathcal{A}\widehat{\otimes} M_{n}.$$

Thus for every $a\in \mathcal{A}$, we have
\begin{align*}
D(a)\otimes\varepsilon_{11}&=(D\otimes1)(a\otimes\varepsilon_{11})\\
&=\lim_{\alpha}((a\otimes\varepsilon_{11})\cdot a_{\alpha}-a_{\alpha}\cdot(a\otimes\varepsilon_{11}))\\
&=\lim_{\alpha}(\sum_{i=1}^{n}aa_{i1}^{\alpha}\otimes
\varepsilon_{i1}-\sum_{j=1}^{n}a_{1j}^{\alpha}a\otimes\varepsilon_{1j}).
\end{align*}

Hence

$$\left(
    \begin{array}{cccc}
      Da & 0 & \ldots &0 \\
      0 & 0 & \ldots & 0 \\
      \vdots & \vdots & & \vdots \\
      0 & 0 & \ldots & 0 \\
    \end{array}
  \right)
=\lim_{\alpha}(\left(
   \begin{array}{cccc}
     aa_{11}^{\alpha} & 0 & \ldots & 0 \\
     aa_{21}^{\alpha} & 0 & \ldots & 0 \\
     \vdots & \vdots&  & \vdots \\
     aa_{n1}^{\alpha} & 0 & \ldots & 0 \\
   \end{array}
 \right)
-\left(
   \begin{array}{cccc}
     a_{11}^{\alpha}a & a_{12}^{\alpha}a & \ldots & a_{1n}^{\alpha}a \\
     0 & 0 & \ldots & 0\\
     \vdots& \vdots &  & \vdots \\
     0 & 0 & \ldots & 0 \\
   \end{array}
 \right))
,$$

The above equality implies that
$D(a)=\lim_{\alpha}aa_{11}^{\alpha}-a_{11}^{\alpha}a$. Therefore
$\mathcal{A}$ is approximately cyclic
amenable.
\end{proof}


\section{Approximate cyclic amenability of second dual}

For a Banach algebra $\mathcal{A}$ let $\mathcal{A}^{op}$ be a Banach algebra, whose underlying Banach space is $\mathcal{A}$, but the product is $\circ$, where $a\circ b=ba$ in which $a,b\in \mathcal{A}$.

\begin{lemma}\label{lee}
Let $\mathcal{A}$ be a Banach algebra.
Then $\mathcal{A}$ is approximately cyclic amenable if and only if
$\mathcal{A}^{op}$ is approximately cyclic amenable.
\end{lemma}
\begin{proof}The identity map $i:\mathcal{A}\longrightarrow
\mathcal{A}^{op}$ is a continuous anti-isomorphism. Suppose that
the map $D:\mathcal{A}^{op}\longrightarrow \mathcal({A}^{op})^{*}$ is
a continuous cyclic derivation. It is easy to see that $i^{*}\circ
D\circ i:\mathcal{A}\longrightarrow \mathcal{A}^{*}$ is a cyclic
derivation. Since $\mathcal{A}$ is approximately cyclic amenable,
there exists a net $(x_{\alpha})\subseteq \mathcal{A}^{*}$ such
that
$$i^{*}\circ D\circ i(a)=\lim_{\alpha}(a\cdot x_{\alpha}-x_{\alpha}\cdot a)\quad (a\in \mathcal{A}).$$

Moreover $(i^{*})^{2}=I_{\mathcal({A}^{op})^{*}}$. Applying $i^{*}$ to both sides of the above equation, we have
$$D(a)=D(i(a))=\lim_{\alpha}(a\cdot x_{\alpha}-x_{\alpha}\cdot a)\quad(a\in \mathcal{A}).$$

This implies that $\mathcal{A}^{op}$ is approximately cyclic
amenable. Since $(\mathcal{A}^{op})^{op}=\mathcal{A}$, the proof
of converse is done similarly.
\end{proof}

We need the following result which has been
proven in \cite[Proposition 4.6]{ess}.

\begin{prop}\label{prp}
Let $\mathcal{A}$ be a Banach
algebra, $\mathcal{B}$ be a closed subalgebra of $\mathcal{A}$ and $\mathcal{I}$ be a closed
ideal are in $\mathcal{A}$ such that
$\mathcal{A}=\mathcal{B}\oplus \mathcal{I}$. If $\mathcal{A}$ is
approximately cyclic amenable, then so is $\mathcal{B}$.
\end{prop}

Recall that a Banach algebra $\mathcal{A}$ is said to be a dual
Banach algebra if there is a closed submodule $\mathcal{A}_{*}$
of $\mathcal{A}^{*}$
 such that $\mathcal{A}=(\mathcal{A}_{*})^{*}.$

\begin{theorem}\label{tht1}
Let $\mathcal{A}$ be a dual Banach
algebra. If $(\mathcal{A}^{**}, \square)$ is approximately cyclic
amenable then $\mathcal{A}$ is approximately cyclic amenable.
\end{theorem}
\begin{proof}
According to \cite[Theorem 2.15]{dal2},
$(\mathcal{A}_{*})^{\bot}$ is a $\omega^{*}$-closed ideal in
$\mathcal{A}^{**}$ and
 $\mathcal{A}^{**}=\mathcal{A}\oplus(\mathcal{A}_{*})^{\bot}$. Since $\mathcal{A}^{**}$ is approximately cyclic amenable,
  $\mathcal{A}$ is approximately cyclic
  amenable by Proposition \ref{prp}.
\end{proof}
 We have shown in Propositions \ref{ppp} (ii) and  \ref{prp} that under certain conditions, the homomorphic image of an approximately cyclic amenable Banach algebra is again an approximately cyclic amenable. In the upcoming theorem,  we extend Proposition \ref{prp} by using homomorphisms.

\begin{theorem}\label{tht2}Let $\mathcal{A}$ and $\mathcal{B}$ be
Banach algebras. If $\varphi:\mathcal{A}\longrightarrow
\mathcal{B}$ and $\psi:\mathcal{B}\longrightarrow \mathcal{A}$
are continuous homomorphisms such that
$\varphi\circ\psi=I_{\mathcal{B}}$, then
\begin{enumerate}
\item[\emph{(i)}] {If $\mathcal{A}$ is approximately cyclic amenable, then so is $\mathcal{B}$;}
\item[\emph{(ii)}] {If $(\mathcal{A}^{**}, \square)$ is approximately cyclic amenable, then so is $(\mathcal{B}^{**}, \square)$.}
\end{enumerate}
\end{theorem}
\begin{proof}(i)  Let $D:\mathcal{B}\longrightarrow
\mathcal{B}^{*}$ be a cyclic derivation. The map  $\varphi^{*}$ is
an $\mathcal{A}$-module homomorphism, and so

\begin{align*}
\varphi^{*}\circ D\circ\varphi(ab) &=\varphi^{*}(D(\varphi(a))\cdot\varphi(b)+\varphi(a)\cdot D(\varphi(b)))\\
&=\varphi^{*}\circ D\circ\varphi(a)\cdot b+a\cdot\varphi^{*}\circ
D\circ\varphi(b),
\end{align*}
for all $a,b\in \mathcal{A}$. Hence, $\varphi^{*}\circ
D\circ\varphi:\mathcal{A}\longrightarrow \mathcal{A}^{*}$ is a
continuous derivation. Moreover, $\varphi^{*}\circ D\circ\varphi$
is cyclic, since $D$ is cyclic. Therefore there exists a net
$(a_{\alpha}^{*})\subseteq \mathcal{A}^{*}$ such that
$\varphi^{*}\circ D\circ\varphi(a)=\lim_{\alpha}(a\cdot
a_{\alpha}^{*}-a_{\alpha}^{*}\cdot a)\quad(a\in \mathcal{A}).$ The
equality $\varphi\circ\psi=I_{\mathcal{B}}$ implies
$\psi^{*}\circ\varphi^{*}=I_{\mathcal{B}^{*}}$. For every
$c\in \mathcal{B}$, we get

\begin{align*}
D(c) &=\psi^{*}\circ\varphi^{*}\circ D\circ \varphi\circ\psi(c)\\
&=\psi^{*}(\varphi^{*}\circ D\circ\varphi(\psi(c))\\
&=\psi^{*}(\lim_{\alpha}(\psi(c)\cdot a_{\alpha}^{*}-a_{\alpha}^{*}\cdot \psi(c)))\\
&=\lim_{\alpha}\psi^{*}(\psi(c)\cdot a_{\alpha}^{*}-a_{\alpha}^{*}\cdot\psi(c))\\
&=\lim_{\alpha}(c\cdot
\psi^{*}(a_{\alpha}^{*})-\psi^{*}(a_{\alpha}^{*})\cdot c).
\end{align*}

It follows that $\mathcal{B}$ is approximately cyclic amenable.

(ii) Since $\varphi^{**}\circ\psi^{**}=I_{\mathcal{B}^{**}}$,
the proof is similar to (i). 
\end{proof}

\begin{theorem}\label{corf}
Let $\mathcal{A}$ and $\mathcal{B}$ be
Banach algebras. Assume that the charater spaces of $\mathcal{A}$ and $\mathcal{B}$ are non empty.
\begin{enumerate}
\item[\emph{(i)}] {If
$\mathcal{A}\widehat{\otimes}\mathcal{B}$ is approximately cyclic
amenable, then so are $\mathcal{A}$ and $\mathcal{B}$;}
\item[\emph{(ii)}] {If $\mathcal{A}$ is commutative and
$\mathcal{A}\widehat{\otimes}\mathcal{A}$ is approximately cyclic
amenable, then $\mathcal{A}$ is approximately cyclic amenable.}
\end{enumerate}
\end{theorem}
\begin{proof}By Proposition \ref{ropo} we may suppose that
$\mathcal{A}$ and $\mathcal{B}$ are unital with identities
$e_{\mathcal{A}}$ and $e_{\mathcal{B}}$.\\ (i)  It is obvious
that the map
$\varphi:\mathcal{A}\widehat{\otimes}\mathcal{B}\longrightarrow
\mathcal{A}$: $a\otimes b\longmapsto \phi(b)a$ is a continuous
linear map, where $\phi\in \Phi_{\mathcal{B}}$. Take $u,v\in
\mathcal{A}\widehat{\otimes}\mathcal{B}$ such that
$u=\sum_{n=1}^{\infty}a_{n}\otimes b_{n}$ and
$v=\sum_{n=1}^{\infty}c_{m}\otimes d_{m}$\quad
$(a_{n},c_{m}\in\mathcal{A}, b_{n},d_{m}\in \mathcal{B})$. Hence,
\begin{align*}
\varphi(uv) &=\varphi(\sum_{n,m=1}^{\infty}a_{n}c_{m}\otimes
b_{n}d_{m})\\
&=\sum_{n,m=1}^{\infty}\phi(b_{n}d_{m})a_{n}c_{m}=\sum_{n,m=1}^{\infty}\phi(b_{n})\phi(d_{m})a_{n}c_{m}\\
&=(\sum_{n=1}^{\infty}\phi(b_{n})a_{n})(\sum_{m=1}^{\infty}\phi(d_{m})c_{m})=
\varphi(u)\varphi(v).
\end{align*}

So, $\varphi$ is a continuous homomorphism. Moreover, the map
$\psi:\mathcal{A}\longrightarrow
\mathcal{A}\widehat{\otimes}\mathcal{B}$\quad $a\longmapsto
a\otimes e_{\mathcal{B}}$ is continuous homomorphism such that
$\varphi\circ\psi=I_{\mathcal{A}}$. As a result of Theorem \ref{tht2} (i), $\mathcal{A}$ is approximately cyclic amenable and likewise $\mathcal{B}$ is approximately cyclic
amenable.

(ii) Consider the homomorphisms
$\varphi:\mathcal{A}\widehat{\otimes}\mathcal{A}\longrightarrow
\mathcal{A}$: $a\widehat{\otimes}b\longmapsto ab$ and
$\psi:\mathcal{A}\longrightarrow
\mathcal{A}\widehat{\otimes}\mathcal{A}$: $a\longmapsto a\otimes
e_{\mathcal{A}}$. Thus $\varphi\circ\psi=I_{\mathcal{A}}$. Now,
Theorem \ref{tht2} (i) shows that $\mathcal{A}$ is approximately cyclic
amenable.
\end{proof}

In the rest of the paper, we investigate conditions under which approximate cyclic amenability of $\mathcal{A}^{**}$
necessitates approximate cyclic amenability of  $\mathcal{A}$.

Let $\mathcal{A}$ be a Banach algebra. The space of almost periodic functionals on
$\mathcal{A}$ is defined by
$$WAP(A)=\{a^*\in \mathcal{A}^{*}:a\longmapsto a\cdot a^*; \,\, \mathcal{A}\longrightarrow \mathcal{A}^{*}\,  \text{is weakly compact}\}.$$
Also, the topological center $Z(\mathcal{A}^{**})$ of
$\mathcal{A}^{**}$ is defined by
$$Z(\mathcal{A}^{**})=\{b^{**} : \text{The map }a^{**}
\mapsto b^{**}\square a^{**} \hspace{0.2cm} \text{is}
\hspace{0.2cm}
 \omega^{*}-\omega^{*}\text{-continuous} \}.$$

\begin{theorem}\label{t11}
Let $\mathcal{A}$ be a Banach algebra.

\begin{enumerate}
\item[\emph{(i)}] {Suppose $(\mathcal{A}^{**}, \square)$ is approximately cyclic amenable and every cyclic derivation $D:\mathcal{A}\longrightarrow\mathcal{A}^{*}$
satisfies $D^{**}(\mathcal{A}^{**})\subseteq WAP(\mathcal{A})$.
Then $\mathcal{A}$ is approximately cyclic amenable;}
\item[\emph{(ii)}] {Suppose $\mathcal{A}$ is Arens regular, every
derivation $D:\mathcal{A}\longrightarrow \mathcal{A}^{*}$ is
weakly compact and $\mathcal{A}^{**}$ is approximately cyclic amenable.
Then $\mathcal{A}$ is approximately cyclic amenable.}
\end{enumerate}
\end{theorem}
\begin{proof}
(i) Let $D:\mathcal{A}\longrightarrow \mathcal{A}^{*}$ be a cyclic derivation. It follows from \cite[Theorem 2.1]{esh} that
$D^{**}:\mathcal{A}^{**}\longrightarrow \mathcal{A}^{***}$ is a
derivation. For every $a^{**},b^{**}\in \mathcal{A}^{**}$ take
two bounded nets $(a_{\alpha}),(b_{\beta})\subseteq \mathcal{A}$
such that $a^{**}=\omega^{*}-\lim_{\alpha}a_{\alpha}$ and
$b^{**}=\omega^{*}-\lim_{\beta}b_{\beta}$. Hence, $\langle
D^{**}(a^{**}),b^{**}\rangle=\lim_{\alpha}\lim_{\beta}\langle
D(a_{\alpha}),b_{\beta}\rangle$ and since
$D^{**}(a^{**})\subseteq WAP(\mathcal{A})\subseteq
\mathcal{A}^{*}$, $\langle
D^{**}(b^{**}),a^{**}\rangle=\lim_{\alpha}\lim_{\beta}\langle
D(b_{\beta}),a_{\alpha}\rangle$. The above equalities show that
$$\langle D^{**}(a^{**}),b^{**}\rangle+\langle D^{**}(b^{**}),a^{**}\rangle=\lim_{\alpha}\lim_{\beta}(\langle D(a_{\alpha}),b_{\beta}\rangle+\langle Db_{\beta},a_{\alpha}\rangle)=0.$$

Thus $D^{**}$ is cyclic, and so there exists a net
$(a_{\alpha}^{***})\subseteq \mathcal{A}^{***}$ such that
$$D^{**}(a^{**})=\lim_{\alpha}(a^{**}\cdot a_{\alpha}^{***}-a_{\alpha}^{***}\cdot a^{**})\quad(a^{**}\in \mathcal{A}^{**}).$$

On the other hand,
$$D(a)=P(D(a))=\lim_{\alpha}(a\cdot P(a_{\alpha}^{***})-P(a_{\alpha}^{***})\cdot a)\quad (a\in \mathcal{A}),$$
where $P:\mathcal{A}^{***}\longrightarrow \mathcal{A}^{*}$ is the
natural projection. Therefore $D$ is approximately inner and thus
$\mathcal{A}$ is approximately cyclic amenable.

(ii) Suppose $D:\mathcal{A}\longrightarrow \mathcal{A}^{*}$ is
a cyclic derivation. Since $D$ is weakly compact,
$D^{**}(\mathcal{A}^{**})\subseteq \mathcal{A}^{*}$ and by
\cite[Corollary 7.2(i)]{dal4},
$D^{**}:\mathcal{A}^{**}\longrightarrow \mathcal{A}^{***}$ is a
derivation. So similar to part (i), $D^{**}$ is cyclic and
approximately inner. Therefore, we obtain the desired
result.
\end{proof}

\begin{theorem}\label{t22}
Let $\mathcal{A}$ be a Banach algebra
and every cyclic derivation $D:\mathcal{A}\longrightarrow
\mathcal{A}^{*}$ is weakly compact. If the topological center $Z(\mathcal{A}^{**})$ is
approximately cyclic amenable, then so is $\mathcal{A}$.
\end{theorem}
\begin{proof}  For simplicity, we put $\mathcal{B}=Z(\mathcal{A}^{**})$. Let $D:\mathcal{A}\longrightarrow
\mathcal{A}^{*}$ be a cyclic derivation and
$J:\mathcal{B}\longrightarrow \mathcal{A}^{**}$ is the inclusion
map. Let us consider the map $\widetilde{D}:=J^{*}\circ D^{**}\circ J:\mathcal{B}\longrightarrow\mathcal{B}^{*}$.
Assume that $a^{**}, b^{**}\in \mathcal{B}$ and choose two
bounded nets $(a_{\alpha}), (b_{\beta})\subseteq \mathcal A$ such that
$a^{**}=\omega^{*}-\lim_{\alpha}a_{\alpha}$ and
$b^{**}=\omega^{*}-\lim_{\beta}b_{\beta}$. Then

\begin{align*}
\widetilde{D}(a^{**}b^{**})
&=\lim_{\alpha}\lim_{\beta}\widetilde{D}(a_{\alpha}b_{\beta})\\
&=\lim_{\alpha}\lim_{\beta}J^{*}(D(a_{\alpha})\cdot b_{\beta}+a_{\alpha}\cdot D(b_{\beta}))\\
&=\lim_{\alpha}\lim_{\beta}(J^{*}\circ D(a_{\alpha})\cdot
b_{\beta}+a_{\alpha}\cdot J^{*}\circ D(b_{\beta})).
\end{align*}

Note that in the last equality we used the fact that $J^{*}$ is a
$\mathcal{B}$-module homomorphism. Obviously,
$\lim_{\alpha}\lim_{\beta}J^{*}\circ D(a_{\alpha})\cdot
b_{\beta}=\widetilde{D}(a^{**})\cdot b^{**}$. Since the map
$a^{**}\mapsto b^{**}a^{**}$, $\omega^{*}-\omega^{*}$-continuous,
$\lim_{\alpha}\lim_{\beta}a_{\alpha}\cdot J^{*}\circ
D(b_{\beta})=a^{**}\cdot\widetilde{D}(b^{**})$ and as a result
$\widetilde{D}$ is a derivation. Moreover, $\langle
\widetilde{D}(a^{**}),
b^{**}\rangle=\lim_{\alpha}\lim_{\beta}\langle
D(a_{\alpha}),b_{\beta}\rangle$. By assumption since D is weakly
compact we have $D^{**}(\mathcal{A}^{**})\subseteq
\mathcal{A}^{*}$ and so $\langle
\widetilde{D}(b^{**}),a^{**}\rangle=\lim_{\alpha}\lim_{\beta}\langle
D(b_{\beta}),a_{\alpha}\rangle$. Hence
\begin{align*}
0&=\lim_{\alpha}\lim_{\beta}(\langle D(a_{\alpha}),b_{\beta}\rangle+\langle D(b_{\beta}),a_{\alpha}\rangle)\\
&=\langle \widetilde{D}(a^{**}),b^{**}\rangle+\langle
\widetilde{D}(b^{**}),a^{**}\rangle.
\end{align*}

Therefore, $\widetilde{D}$ is a cyclic derivation and since $\mathcal{B}$ is approximately cyclic amenable then there
exists a net $(b_{\gamma}^{*})\subseteq \mathcal{B}^{*}$ such that
$$\widetilde{D}(b)=\lim_{\gamma}(b\cdot b_{\gamma}^{*}-b_{\gamma}^{*}\cdot b)\quad (b\in \mathcal{B}).$$

Now, set $a_{\gamma}^{*}=b_{\gamma}^{*}|_{\mathcal{A}}$. By
considering the net $(a_{\gamma}^{*})\subseteq \mathcal{A}^{*}$
and above relation we have
$$D(a)=\lim_{\gamma}(a\cdot a_{\gamma}^{*}-a_{\gamma}^{*}\cdot a)\quad (a\in \mathcal{A}).$$

Therefore $\mathcal{A}$ is approximately cyclic amenable.
\end{proof}

\section*{Acknowledgement}
The authors sincerely thank the anonymous reviewer for his
careful reading, constructive comments and fruitful suggestions
to improve the quality of the paper. The first author was supported by Islamic Azad
University, Karaj Branch. He would like to
thanks this University for their kind support.

\end{document}